\documentclass[12pt]{article}
\usepackage{amsmath}
\usepackage{amsthm,amscd,amsfonts}
\usepackage{amssymb, upref, color}
\usepackage{amsmath,amssymb,amsthm}
\usepackage{color}
\usepackage[colorlinks]{hyperref}
\usepackage{graphicx}
\usepackage{epsf,epsfig,subfigure, verbatim}
\usepackage{latexsym,bm}
\usepackage{enumerate}
\usepackage{listings}
\usepackage[numbers,sort&compress]{natbib}
\usepackage{mathrsfs}
\usepackage{color}

\newtheorem{thm}{Theorem}[section]
\newtheorem{lem}{Lemma}[section]
\newtheorem{rmk}{Remark}[section]

\newtheorem{defn}{Definition}[section]
\newtheorem{example}{Example}[section]

\newtheorem{cor}{Corollary}[section]

\begin{document}
\pagenumbering{arabic}

\title{Floquet multipliers and the stability of periodic linear
differential equations: a unified algorithm and its computer realization
\footnote{ This paper was supported by the National Natural
Science Foundation of China under Grant (No. 11931016 and 11671176).}}
\author{Mengda Wu, Yonghui Xia\,\footnote{ Corresponding author. Yonghui Xia, xiadoc@163.com;yhxia@zjnu.cn.  } , Ziyi Xu
  \\
{\small \em College of Mathematics and Computer Science, Zhejiang Normal University, Jinhua, 321004, China}\\
{\small \em  yhxia@zjnu.cn; medawu@zjnu.edu.cn;  ziyixu@zjnu.edu.cn   }
}
\date{}
\maketitle \pagestyle{myheadings} \markboth{}{}

\noindent
\begin{abstract}
Floquet multipliers (characteristic multipliers) play significant role in the stability of the periodic equations. Based on the iterative method, we provide a unified algorithm to compute the Floquet multipliers (characteristic multipliers) and determine the stability of the periodic linear differential equations on time scales unifying discrete, continuous, and hybrid dynamics. Our approach is based on calculating the value of $\mathcal{A}$ and $\mathcal{B}$ (see Theorem \ref{thm12.111}), which are the sum and product of all Floquet multipliers (characteristic multipliers) of the system, respectively.
We obtain an explicit expression of $\mathcal{A}$ (see Theorem \ref{thm12.112}) by the method of variation and approximation theory (iterative method), and an explicit expression of $\mathcal{B}$ by Liouville's formula.
Furthermore, a computer program is designed to realize our algorithm. Specifically,  you can determine the stability of a second order periodic linear system, whether they are discrete, continuous or hybrid, as long as you enter the program codes associated with the parameters of the equation.  In fact, few literatures have dealt with the algorithm to compute the Floquet multipliers, not mention to design the program for its computer realization. Our algorithm gives the explicit expressions of all Floquet multipliers and our computer program is based on the approximations  of these explicit expressions.
In particular, on an arbitrary discrete periodic time scale, we can do a finite number of calculations to get the explicit value of Floquet multipliers 
(see Theorem \ref{thm12.113}). Therefore, for any discrete periodic system, we can accurately determine the stability of the system by our algorithm even without computer!  Finally, in Section 6, several examples are presented to illustrate the effectiveness of our algorithm.
\end{abstract}



\section{Introduction}

\subsection{Theory of time scales unifying the continuous and discrete calculus}

In 1988, Hilger \cite{1988} introduced the theory of time scales for the propose of unifying discrete and continuous calculus (\cite{chain,unified}). The systematic works of dynamic equations on time scales, one can refer to  Bohner and Peterson \cite{ts}, Agarwal and Bohner \cite{10181}, Agarwal et al \cite{10182}, and Bohner et al. \cite{advance}. In particular, the theory of the exponential dichotomy, reducibility, linearization,  Hyers-Ulam stability and Sturmain theory are well studied, one can refer to P\"otzche \cite{Potz1,Potz2,Potz3,Potz4} and Siegmund  \cite{10186}, Doan et al. \cite{10187,10188,10189}, Zhang et al \cite{ZhangJM1,ZhangJM2}, Reinfelds and $\breve{S}$teinberga \cite{Reinfelds}, Erbe and Peterson \cite{10184}, Erbe and Hilger \cite{10185}.
It was also generalized to the measure differential equations on time scales (Federson et al. \cite{F-JDE1,F-JDE3}), and fuzzy-valued differential equations on time scales (Wang et al. \cite{WangC3,WangC4,WangC1,WangC5,WangC6}), quaternion-valued differential equations on time scales (Li et al. \cite{WangC2}).
Recently, DaCunha and Davis \cite{duu}, DaCunha \cite{duphd} extend the Floquet theory to a more general case of an arbitrary periodic time scale which unifies discrete, continuous, and hybrid periodic cases.
Adivar and Koyuncuo\u{g}lu \cite{20221} constructs a unified Floquet theory for homogeneous and nonhomogeneous hybrid periodic systems on domains having continuous, discrete or hybrid structure using the new periodicity concept based on shifts.

\subsection{Floquet theory and Floquet multipliers}
Floquet theory indicates that  a nonautonomous $T$-periodic linear system of differential equations can be reducible to a corresponding autonomous linear system of differential equations by a periodic Lyapunov transformation \cite{d8}.
Floquet theory is a powerful
tool to study the stability and  periodic solutions of dynamic systems. 
Mathematicians have extended Floquet theory in different directions. We can classify the results of Floquet theory into some types: ODEs (almost Floquet systems \cite{d12}, almost-periodic systems \cite{d19}, periodic Euler-Bernoulli equations \cite{d25}, delay differential equations \cite{d29}, linear systems with meromorphic solutions \cite{d31}), PDEs (parabolic differential equations \cite{d9}, periodic evolution problems \cite{d21}), DAEs \cite{d11,d22}, integro-differential equations \cite{20222}, Volterra equations \cite{20223}, discrete dynamical systems (countable systems \cite{d30}) and systems on time scales \cite{d2}. More details for the Floquet theory and applications, one can also refer to the monograph \cite{d20} and the works \cite{d13,20225}.

Floquet multipliers (characteristic multipliers) play great role in the Floquet theory and the stability of the periodic equations. Thus, usually, to determine the stability, it suffices to calculate the characteristic multipliers. More specifically,  if all of the characteristic multipliers have modulus less than or equal to one, and if, for each characteristic multiplier with modulus equal to one, the algebraic multiplicity equals the geometric multiplicity, the system is stable,  otherwise the system is unstable. Then a natural question is how to compute the characteristic multipliers of the periodic systems. To this end, mathematicians have proposed some methods to compute the characteristic multipliers of periodic differential equations. For examples, Kotsis \cite{172} studied the approximation of the characteristic multipliers based on a theorem of Demidovi$\breve{c}$; Shi \cite{tableshi} estimated the periodic Hill equation; some very nice results were obtained for the delay differential equations (functional differential equations), see Breda, Mast and Vermiglio \cite{171}, Chow and Walther \cite{Chow}), Val'ter and  Skubachevskii \cite{162}, Skubachevskii and Walther \cite{164}), Walther \cite{169,170}, Luzyanina and Engelborghs \cite{163}, Dormayer et al. \cite{165}
 Huang and  Mallet-Paret \cite{166}, Mallet-Paret and Sell \cite{168}.

\subsection{Motivation, novelty and contributions}

There are a few works considering the Floquet theory and characteristic multipliers as mentioned above.
{\bf However, few existing literatures have dealt with the algorithm to compute the Floquet multipliers (characteristic multipliers), not mention to design the program for its computer realization.}
In this paper, we provide a unified algorithm to compute the Floquet multipliers (characteristic multipliers) and determine the stability of the second order periodic linear equations on periodic time scales. 
   {We claim that the periodic system  is stable if \begin{small}$$\left|\frac{\mathcal{A}}{2}+\sqrt{(\frac{\mathcal{A}}{2})^2-\mathcal{B}}\right|<1\quad \text{and}\quad\left|\frac{\mathcal{A}}{2}-\sqrt{(\frac{\mathcal{A}}{2})^2-\mathcal{B}}\right|<1,$$\end{small}and system is unstable if\begin{small} $$\left|\frac{\mathcal{A}}{2}+\sqrt{(\frac{\mathcal{A}}{2})^2-\mathcal{B}}\right|>1\quad \text{or}\quad\left|\frac{\mathcal{A}}{2}-\sqrt{(\frac{\mathcal{A}}{2})^2-\mathcal{B}}\right|>1.$$\end{small}}
  To determine the stability of the periodic system, it is sufficient to know the modulus of characteristic multipliers, which can be derived from $\mathcal{A}$ and $\mathcal{B}$. Our main task is to calculate the value of $\mathcal{A}$ and $\mathcal{B}$  (see Theorem \ref{thm12.111}), which are the sum and product of all characteristic multipliers of the system, respectively.  We obtain an {explicit expression of $\mathcal{A}$} (see Theorem \ref{thm12.112}) by the method of variation and approximation theory (iterative method) and {an explicit expression of $\mathcal{B}$} by Liouville's formula. Finally, in Section 6, several examples are presented to illustrate the effectiveness of our algorithm. {  The illustrative examples show good performance of our computer program.}  We summarize the contributions of this paper as follows.\\
\noindent (1) Based on the iterative method, we provide {\bf a unified algorithm} to compute Floquet multipliers and determine the stability of the periodic linear differential equations on time scales
   unifying discrete, continuous, and hybrid dynamics.\\
\noindent (2){\bf A computer program is designed} to realize our algorithm. {Therefore,  you can determine the stability of a second order periodic linear system, whether they are discrete, continuous or hybrid, as long as you enter the program codes associated with the parameters of the equation.  }\\
\noindent (3) Few existing literatures have dealt with the algorithm to compute the Floquet multipliers, { not mention to design the program for its computer realization.} Our algorithm gives the {\bf explicit expressions} of all Floquet multipliers and our computer program is based on the approximations of these explicit expressions.\\
\noindent (4)  We provide {\bf an estimate of the error between $\mathcal{A}(n)$ and $\mathcal{A}$}. And a computer program  is given for calculating the value of $\mathcal{A}(n)$, $\mathcal{B}$ and $\rho(n)$, where $\mathcal{A}(n)$ is the $n$-th approximation of $\mathcal{A}$ and $\rho(n)$ is the $n$-th approximations of modulus of characteristic multipliers.\\
\noindent (5)  In particular,  {on an arbitrary discrete periodic time scale}, there is a constant $k\in\mathbb{N}$, such that $\mathcal{A}=\mathcal{A}(k)$. Consequently, we can do a finite number of calculations to get the explicit value of Floquet multipliers  (see Theorem \ref{thm12.113}). Therefore, for any discrete periodic system, we can  accurately determine the stability of the system by our algorithm even without computer!\\
\noindent (6) 
 We obtain an {\bf explicit expression of $\mathcal{A}$} (see Theorem \ref{thm12.112}) by the method of variation and approximation theory and {\bf an explicit expression of $\mathcal{B}$} by Liouville's formula.

\subsection{Outline of the paper}

     The rest of this paper is organized as follows. In Section 2, we introduce some notations and lemmas.  Section 3 gives the stability criteria for the systems we studied. Section 4 introduces the processes of getting the expression of $\mathcal{A}$. Our main results on the expression of $\mathcal{A}$ are collected in three theorems (Theorem \ref{thm12.112}--Theorem \ref{thm12.114}). In Section 5, a computer program is given. Finally, in Section 6, we give some examples to show the effectiveness of our  algorithm and verify our computer program.

\section{Preliminaries}
For completeness, we recall the following notations and concepts for the theory of time scales from \cite{ts}. A time scale $\mathbb{T}$ is a nonempty closed subset of $\mathbb{R}$. We denote $[a,b]\cap\mathbb{T}$ by $[a,b]_{\mathbb{T}}$. The forward jump operator is defined by $\sigma(t):=\inf\{s\in\mathbb{T}:s>t\}$. The backward jump operator is defined by $\rho(t):=\sup\{s\in\mathbb{T}:s<t\}$.  We put $\inf\emptyset=\sup\mathbb{T}$ and $\sup\emptyset=\inf\mathbb{T}$. A point $t\in\mathbb{T}$ is said to be right-dense if $\sigma(t)=t$, right-scattered if $\sigma(t)>t$, left-dense if $\rho(t)=t$, left-scattered if $\rho(t)<t$, isolated if $\rho(t)<t<\sigma(t)$, and dense if $\rho(t)=t=\sigma(t)$. A set $\mathbb{T}^{\kappa}$ is defined as $\mathbb{T}^{\kappa}=\mathbb{T}-\{m\}$ if $\mathbb{T}$ has a left-scattered maximum, $\mathbb{T}^{\kappa}=\mathbb{T}$ otherwise. A time scale $\mathbb{T}$ is said to be discrete if $t$ is scattered for all $t\in\mathbb{T}$, and it is said to be continuous if $t$ is dense for all $t\in\mathbb{T}$. A function $f:\mathbb{T}\rightarrow\mathbb{R}$ is called regulated provided its right-sided limits exist (finite) at all right-dense points in $\mathbb{T}$ and its left-sided limits exist (finite) at all left-dense points in $\mathbb{T}$.  A function $f:\mathbb{T}\rightarrow\mathbb{R}$ is called  rd-continuous  provided it is continuous at right-dense points in $\mathbb{T}$ and its left-sided limits exist (finite) at left-dense points in $\mathbb{T}$.  The set of rd-continuous functions $f:\mathbb{T}\rightarrow\mathbb{R}$ is denoted by $C_{rd}(\mathbb{T},\mathbb{R})$.
The graininess function $\mu$ is defined by $\mu(t):=\sigma(t)-t$.

We say that a function $p:\mathbb{T}\rightarrow\mathbb{R}$ is regressive provided $1+\mu(t)p(t)\neq 0$ holds for all $t\in\mathbb{T}^{\kappa}$. The set of all regressive and rd-continuous functions $f:\mathbb{T}\rightarrow\mathbb{R}$ is denoted by $\mathcal{R}$. The delta derivative of a function $f:\mathbb{T}\rightarrow\mathbb{R}$ at a point $t\in\mathbb{T}^\kappa$, denoted by $f^{\Delta}(t)$, is defined by $$f^{\Delta}(t)=\lim_{s\searrow \mu(t)}\frac{f(t+s)-f(t)}{s}.$$For a function  $f:\mathbb{T}\rightarrow\mathbb{R}$ we shall talk about the second derivative $f^{\Delta\Delta}$ provided $f^{\Delta}$ is differentiable on $(\mathbb{T^\kappa})^\kappa$ with $f^{\Delta\Delta}=(f^\Delta)^\Delta:(\mathbb{T}^\kappa)^\kappa\rightarrow \mathbb{R}$.
A continuous function $f:\mathbb{T}\rightarrow\mathbb{R}$ is called pre-differentiable with $D$, provided $D\subseteq\mathbb{T}^\kappa,\mathbb{T}^\kappa\backslash D $ is countable and contains no right-scattered elements of $\mathbb{T}$, and $f$ is differentiable at each $t\in D$. A pre-differentiable function $F:\mathbb{T}\rightarrow\mathbb{R}$ is called the pre-antiderivative of $f$ if $F^\Delta(t)=f(t)$ holds for all $t\in D$, where $D$ is the region of differentiation. Then we define the Cauchy integral by\begin{equation*}
	\int_r^sf(t)\Delta t=F(s)-F(r),\quad\textrm{for }r,s\in\mathbb{T},
\end{equation*}where $F$ is the pre-antiderivative of $f$.

 If $p\in\mathcal{R}$, we define the exponential function by $$e_p(t,s)=\exp\left(\int_s^t\lim\limits_{s\searrow\mu(\tau)}\frac{\mathrm{Log}(1+p(\tau)s)}{s}\Delta\tau\right)\quad \textrm{for }s,t\in\mathbb{T}.$$
One can see that $e_p(t,s)$ is a solution of the equation $x^{\Delta}=p(t)x$. The delta derivative of a vector-valued (matrix-valued) function is given by taking the derivative of each components.  The integral of a vector-valued (matrix-valued) function can be given in a similar manner.
Let $A$ be an $m\times n$-matrix-valued function on $\mathbb{T}$. We say that $A$ is rd-continuous on $\mathbb{T}$ if each entry of $A$ is rd-continuous on $\mathbb{T}$, and the class of all such rd-continuous $m\times n$-matrix-valued functions on $\mathbb{T}$ is denoted by $C_{rd}=C_{rd}(\mathbb{T},\mathbb{R}^{m\times n})$. An $n\times  n$-matrix-valued function $A$ on a time scale $\mathbb{T}$ is called regressive provided $I+\mu(t)A(t)$ is invertible for all $t\in\mathbb{T}^{\kappa}$, and the class of all such regressive and rd-continuous functions is denoted by $\mathcal{R}$.
\begin{defn}
	(\cite{ts},p.92) If $p\in C_{rd}$ and $\mu p^2\in\mathcal{R},$ then we define the trigonometric functions $\cos_p$ and $\sin_p$ by$$\cos_p=\frac{e_{ip}+e_{-ip}}{2}\quad and \quad \sin_p=\frac{e_{ip}-e_{-ip}}{2i}.$$
\end{defn}
For trigonometric functions on time scales, we have some formulas, which can be found in (\cite{ts}, Exercise 3.27).

\begin{defn}(\cite{duu})
	Let $T\in(0,\infty)$. Then the time scale $\mathbb{T}$ is {\em T-periodic} if for all $t\in\mathbb{T}$,\begin{enumerate}
		\item $t\in\mathbb{T}$ implies $t+T\in\mathbb{T}$;
		\item $\mu(t)=\mu(t+T).$
	\end{enumerate}
\end{defn}
\begin{defn}(\cite{duu})
	$A:\mathbb{T}\rightarrow\mathbb{R}^{n\times n}$ is {\em T-periodic} if $A(t)=A(t+T)$ for all $t\in\mathbb{T}$.
\end{defn}
Consider the regressive time varying linear dynamic initial value problem\begin{equation}\label{eqe1}
	x^{\Delta}(t)=A(t)x(t),\qquad x(t_0)=x_0,
\end{equation}where $A(t)$ is $T$-periodic for $t\in\mathbb{T}$ and the time scale $\mathbb{T}$ is also $T$-periodic.
\begin{defn}(\cite{duu})
	Let $x_0\in\mathbb{R}^n$ be a nonzero vector and $\Psi(t)$ be any fundamental matrix for the system $(\ref{eqe1})$. The vector solution of the system with initial condition $x(t_0)=x_0$ is given by $\Phi_{A}(t,t_0)x_0.$ The operator $M:\mathbb{R}^n\rightarrow\mathbb{R}^n$ given by $M(x_0):=\Phi_A(t_0+T,t_0)=\Psi(t_0+T)\Psi^{-1}(t_0)x_0,$
	is called a monodromy operator. The eigenvalues of the monodromy operator are called the Floquet (or characteristic) multipliers of the system $(\ref{eqe1})$.
\end{defn}
\begin{lem}\label{Flo}(\cite{duu}, Corollary~7.10)
	Consider the $p$-periodic system $(\ref{eqe1})$.\begin{enumerate}
		\item If all the Floquet multipliers have modulus less than one, then the system $(\ref{eqe1})$ is exponentially stable.
		\item If all of the Floquet multipliers have modulus less than or equal to one, and if, for each Floquet multiplier with modulus equal to one, the algebraic multiplicity equals the geometry multiplicity, then the system $(\ref{eqe1})$ is stable; otherwise the system $(\ref{eqe1})$ is unstable, growing at rates of generalized polynomials of t.
		\item If at least one Floquet multiplier has modulus greater than one, then the system $(\ref{eqe1})$ is unstable.
	\end{enumerate}
\end{lem}
\begin{lem}
	\label{bound}(\cite{ts},p.23) Every regulated function on a compact interval is bounded.
\end{lem}
\begin{lem}
	\label{yizhi}Assume that $D$ is a compact subset of $\mathbb{R}$ and $f_n\in C_{rd}(D,\mathbb{R})$ for each $n\in\mathbb{N}$. If $\{f_n\}$ uniformly converges to $f$ on $D$, then f is rd-continuous and$$\int_a^bf(t)\Delta t=\lim\limits_{n\rightarrow\infty}\int_a^bf_n(t)\Delta t.$$where $a,b\in D$.
\end{lem}
\begin{lem}
\label{budeng}
Let $\mathbb{T}$ be an arbitrary time scale. Suppose $f:[a,b]\rightarrow\mathbb{R}$ is an increasing function, where $a,b\in\mathbb{T}$ (b may be $\infty$). If f is rd-continuous when it is restricted on $[a,b]_{\mathbb{T}}$, then we have $$\int_a^bf(s)\mathrm{d}s\geq\int_a^bf(s)\Delta s.$$
\end{lem}
\begin{proof}
Note that $f$ is an increasing function on $[a,b]$, hence $f$ is integrable on $[a,b]$. Let $\varepsilon>0$. We now show by induction that $$S(t):\quad \int_a^tf(s)+\varepsilon \mathrm{d}s-\int_a^tf(s)\Delta s\geq 0$$
	holds for all $t\in[a,b]_{\mathbb{T}}$.\begin{enumerate}
		\item The statement $S(a)$ is trivially satisfied.
		\item Let t be right-scattered and assume that $S(t)$ holds. Then we have\begin{equation*}
			\begin{split}
				&\int_a^{\sigma(t)}f(s)+\varepsilon \mathrm{d}s-\int_a^{\sigma(t)}f(s)\Delta s\\\geq & \int_t^{\sigma(t)}f(s)+\varepsilon \mathrm{d}s-\int_t^{\sigma(t)}f(s)\Delta s
				\geq \int_t^{\sigma(t)}f(t)+\varepsilon \mathrm{d}s-\mu(t)f(t)=\mu(t)\varepsilon>0.
							\end{split}
		\end{equation*}Therefore $S(\sigma(t))$ holds.
		\item Assume that S(t) holds and $t\neq a$ is right-dense. Since $f(t)\in C_{rd}([a,b]_{\mathbb{T}},\mathbb{R})$, $f(t)$ is continuous (on $\mathbb{T}$) at $t$. Then there exists $\delta=\delta(\varepsilon,t)$, such that $|f(s)-f(t)|\leq\varepsilon/2$ holds for all $s\in(t-\delta,t+\delta)_{\mathbb{T}}$. Hence we have for all $\tau\in(t,t+\delta)_{\mathbb{T}}$,\begin{equation*}
			\begin{split}
				&\int_a^{\tau}f(s)+\varepsilon \mathrm{d}s-\int_a^{\tau}f(s)\Delta s\\\geq & \int_t^{\tau}f(s)+\varepsilon \mathrm{d}s-\int_t^{\tau}f(s)\Delta s
				\geq  (\tau-t)(\varepsilon+f(t)-f(\tau))\geq \frac{\varepsilon(\tau-t)}{2}>0.
							\end{split}
		\end{equation*}Therefore $S(\tau)$ holds for all $\tau\in(t,t+\delta)_{\mathbb{T}}$.
		\item Now let $t$ be left-dense and suppose $S(\tau)$ is true for all $\tau\in[a,t)_{\mathbb{T}}$, then $S(t)$ holds since the function $$F(t,\varepsilon):=\int_a^tf(s)+\varepsilon \mathrm{d}s-\int_a^tf(s)\Delta s$$ is continuous (on $\mathbb{T}$) with respect to $t$.
	\end{enumerate}
	By induction principle (\cite{ts},p.4), $S(b)$ is true (i.e. $F(b,\varepsilon)\geq 0$). Moreover, it can be seen that $F(b,\varepsilon)$ is continuous with respect to $\varepsilon$, then $F(b,0)=\lim\limits_{\varepsilon\rightarrow 0^+}F(b,\varepsilon)\geq 0$. The proof is completed.
\end{proof}
\begin{cor}
Let $\mathbb{T}$ be an arbitrary time scale. Suppose $f:[a,b]\rightarrow\mathbb{R}$ is a decreasing function, where $a,b\in\mathbb{T}$ (b may be $\infty$). If f is rd-continuous when it is restricted on $[a,b]_{\mathbb{T}}$,  then we have $$\int_a^bf(s)\mathrm{d}s\leq\int_a^bf(s)\Delta s.$$
\end{cor}
\begin{cor}\label{cor624}
	Let $\mathbb{T}$ be an arbitrary time scale and $c$ be an arbitrary nonnegative constant. Then we have $$ \int_a^b\int_a^{t_1}\cdots\int_a^{t_{n-1}}c~\Delta t_n\cdots\Delta t_1\leq \frac{c(b-a)^n}{n!},$$ where $a,b\in\mathbb{T},a\leq t_{n-1}\leq\cdots\leq t_1\leq b.$
\end{cor}
\begin{proof}
	Let $b=t_0$. We now show by induction that $$S(k):\quad \int_a^{t_{n-k}}\cdots\int_a^{t_{n-1}}c~\Delta t_n\cdots\Delta t_{n-k+1}\leq \frac{c(t_{n-k}-a)^k}{k!}$$
	holds for all $k\in\{1,2,\ldots,n\}$\begin{enumerate}
		\item Clearly, $S(1)$ holds.
		\item Now suppose $k\leq n-1$ and that $S(k)$ holds. Then\begin{equation*}
			\begin{split}
				&\int_a^{t_{n-(k+1)}}\cdots\int_a^{t_{n-1}}c~\Delta t_n\cdots\Delta t_{n-k}\\
				\leq & \int_a^{t_{n-(k+1)}}\frac{c(t_{n-k}-a)^k}{k!}\Delta t_{n-k}
				\leq  \int_a^{t_{n-(k+1)}}\frac{c(t_{n-k}-a)^k}{k!}\mathrm{d} t_{n-k}=\frac{c(t_{n-(k+1)}-a)^{k+1}}{(k+1)!}.
			\end{split}
		\end{equation*}Thus, $S(k+1)$ holds.
	\end{enumerate}By induction principle, the proof is completed.
\end{proof}
\begin{cor}
	Let $\mathbb{T}$ be an arbitrary time scale and $c$ be an arbitrary nonpositive constant. Then we have $$ \int_a^b\int_a^{t_1}\cdots\int_a^{t_{n-1}}c~\Delta t_n\cdots\Delta t_1\geq \frac{c(b-a)^n}{n!},$$ where $a,b\in\mathbb{T},a\leq t_{n-1}\leq\cdots\leq t_1\leq b.$
\end{cor}
\section{Stability Criteria}
Now we start our main work. Let $\mathbb{T}$ be a $T$-periodic time scale and unbounded above. Consider the stability of the regressive time varying linear dynamic system  \begin{equation}\label{a1}
	x^{\Delta\Delta}+p(t)x^{\Delta}+q(t)x=0,
\end{equation}where $p(t+T)=p(t),~q(t+T)=q(t),~p(t),q(t)\in C_{rd}(\mathbb{T},\mathbb{R}),~1-\mu(t)p(t)+\mu^2(t)q(t)\neq0,~q(t)\neq 0$ for all $t\in\mathbb{T}$. We assume that $ q(t)>0$ if $t$ is right-dense, and the equation \begin{equation}\label{phi}
	x^{\sigma}x=q(t)\end{equation}has a  solution $\phi(t)\in C_{rd}^1(\mathbb{T,R})$, where $x^\sigma$ denotes $x(\sigma(t))$.
\begin{rmk}
	The assumption that Eq. $(\ref{phi})$ exists a  solution $\phi(t)\in C_{rd}^1(\mathbb{T,R})$ can be satisfied for some time scales, such as  discrete time scales, continuous time scales and the combination of them.
\end{rmk}

Note that Eq. $(\ref{a1})$ can be written in the form\begin{equation}\label{a2}
	\left(\begin{array}{c}
		x^{\Delta}\\
		y^{\Delta}
	\end{array}\right)=\left(\begin{array}{cc}
		0 & 1\\
		-q(t) & -p(t)
	\end{array}\right)\left(\begin{array}{c}
		x\\
		y
	\end{array}\right).	\end{equation}
	We assume that $S(t)=\left(\!\!\begin{array}{cc}
		0 & 1\\
		-q(t) & -p(t)
	\end{array}\!\!\right)$ and $Y(t)=\left(\!\!\begin{array}{cc}
		x(t) & \bar{x}(t)\\
		y(t) & \bar{y}(t)
	\end{array}\!\!\right)=\Phi_S(t,t_0)$, then the eigenvalues of $Y(t_0+T)$ are the characteristic multipliers of $(\ref{a2})$. It can be seen that$$\det Y(t_0+T)=e_{-p+\mu q}(t_0+T,t_0)\det Y(t_0)=e_{-p+\mu q}(t_0+T,t_0).$$
Let $\rho_1,\rho_2$ denote the characteristic multipliers of $(\ref{a2})$ and\begin{equation}
	\label{eqdab}\begin{split}\mathcal{A}&=x(t_0+T)+\bar{y}(t_0+T),\\
		\mathcal{B}&=e_{-p+\mu q}(t_0+T,t_0).
	\end{split}
\end{equation} Hence $\rho_1,\rho_2$ satisfy
$$\rho^2-\mathcal{A}\rho+\mathcal{B}=0.$$
Obviously, \begin{equation}
	\label{eq118}\rho_{1,2}=\frac{\mathcal{A}}{2}\pm\sqrt{(\frac{\mathcal{A}}{2})^2-\mathcal{B}}.
\end{equation}
Note that the value of $\mathcal{B}$ can be easily calculated, then if we can get the value of $\mathcal{A}$, the stability of system  $(\ref{a1})$ can be studied by Lemma $\ref{Flo}$.
\begin{thm}\label{thm12.111}
	We claim that system  $(\ref{a1})$ is stable if $$\left|\frac{\mathcal{A}}{2}+\sqrt{(\frac{\mathcal{A}}{2})^2-\mathcal{B}}\right|<1\quad \text{and}\quad\left|\frac{\mathcal{A}}{2}-\sqrt{(\frac{\mathcal{A}}{2})^2-\mathcal{B}}\right|<1,$$and system  $(\ref{a1})$ is unstable if $$\left|\frac{\mathcal{A}}{2}+\sqrt{(\frac{\mathcal{A}}{2})^2-\mathcal{B}}\right|>1\quad \text{or}\quad\left|\frac{\mathcal{A}}{2}-\sqrt{(\frac{\mathcal{A}}{2})^2-\mathcal{B}}\right|>1.$$
\end{thm}
\begin{thm}
\label{rmkab2}
	Assume that $\mathcal{B}=1$. Then we have\begin{enumerate}
		\item if $|\mathcal{A}|<2$, system $(\ref{a1})$ is stable;
		\item if   $|\mathcal{A}|>2$, system $(\ref{a1})$ is unstable.
	\end{enumerate}
\end{thm}
\begin{proof}
	It follows from $(\ref{eq118})$ that $|\rho_1|=|\rho_2|=1$ and $\rho_1\neq\rho_2$ as $|\mathcal{A}|<2$, $\mathcal{B}=1$, which implies that  system $(\ref{a1})$ is stable.  The proof of (ii) is similar.
\end{proof}
	\begin{rmk}
		If $\mathbb{T}=\mathbb{R}$, system $(\ref{a1})$ reduces to $x''+p(t)x'+q(t)x=0$. If $\mathbb{T}=\mathbb{Z}$, system $(\ref{a1})$ reduces to $\Delta\Delta x+p(t)\Delta x+q(t)x=0$. In fact, the explicit expression of $\mathcal{A}$ is important to study the stability of the system. Thus, the next section is devoted to presenting an algorithm for the expression of $\mathcal{A}.$
	\end{rmk}
\section{Algorithm for the Expression of  $\mathcal{A}$}
In this section, we are going to focus on the algorithm  for $\mathcal{A}$. Note that system $(\ref{a2})$ can be written as
\begin{equation}\label{eq3.1}
	\left(\begin{matrix}
		x^{\Delta}\\
		y^{\Delta}
	\end{matrix}\right)=\left(\begin{matrix}
		0 & 1\\&\\
		-q(t) & \dfrac{\phi^{\Delta}(t)}{\phi(t)}
	\end{matrix}\right)\left(\begin{matrix}
		x\\
		y
	\end{matrix}\right)+\left(\begin{matrix}
		0\\&\\
		(-p(t)-\dfrac{\phi^{\Delta}(t)}{\phi(t)})y
	\end{matrix}\right).
\end{equation}
Let \begin{equation}\label{eq3.2}
	h(t)=-p(t)-\frac{\phi^{\Delta}(t)}{\phi(t)},
	\end{equation}
thus Eq. $(\ref{eq3.1})$ can be rewritten as \begin{equation}\label{eq3.3}
	\left(\begin{matrix}
		x^{\Delta}\\
		y^{\Delta}
	\end{matrix}\right)=\left(\begin{matrix}
		0 & 1\\&\\
		-q(t) & \dfrac{\phi^{\Delta}(t)}{\phi(t)}
	\end{matrix}\right)\left(\begin{matrix}
		x\\
		y
	\end{matrix}\right)+\left(\begin{matrix}
		0\\&\\
		h(t)y
	\end{matrix}\right).
\end{equation}
Let $\cos_{\phi}(t,t_0)=\cos_{\phi}(t),\sin_{\phi}(t,t_0)=\sin_{\phi}(t)$, hence it can be verified that \begin{equation}\label{eq3.5}
	X(t)=\left(\begin{matrix}
		\cos_{\phi}(t)&\dfrac{1}{\phi(t_0)}\sin_{\phi}(t)\\&\\
		-\phi(t)\sin_{\phi}(t)&\dfrac{\phi(t)}{\phi(t_0)}\cos_{\phi}(t)
	\end{matrix}\right)
\end{equation}
is the fundamental matrix solution of the system
\begin{equation}\label{eqb}
	\left(\begin{matrix}
		x^{\Delta}\\
		y^{\Delta}
	\end{matrix}\right)=\left(\begin{matrix}
		0 & 1\\&\\
		-q(t) & \dfrac{\phi^{\Delta}(t)}{\phi(t)}
	\end{matrix}\right)\left(\begin{matrix}
		x\\
		y
	\end{matrix}\right).
\end{equation}
\begin{rmk}
   Let $A(t)=\left(\begin{smallmatrix}
		0 & 1\\&\\
		-q(t) & \frac{\phi^{\Delta}(t)}{\phi(t)}
	\end{smallmatrix}\right)$ and we claim that $A(t)\in\mathcal{R}$. On the one hand, $q(t),\phi^{\Delta}(t)$ are rd-continuous and $\phi(t)\neq 0$, so $A(t)\in C_{rd}(\mathbb{T},\mathbb{R}^{2\times 2})$. On the other hand, $$\det(I+\mu(t) A(t))=\dfrac{\phi^{\sigma}(t)}{\phi(t)}+\mu^2(t)q(t)=\frac{\phi^{\sigma}(t)(1+\mu^2(t)\phi^2(t))}{\phi(t)}\neq 0,\quad \text{for all }t\in\mathbb{T},$$ hence $A(t)$ is regressive. Besides we have to consider the rationality of the function $\sin_{\phi}(t)$ and $\cos_{\phi}(t)$. We assert that $\sin_{\phi}(t)$ and $\cos_{\phi}(t)$ are well defined, since $$(1+i\mu(t)\phi(t))(1-i\mu(t)\phi(t))=1+\mu^2(t)\phi^2(t)\neq 0$$ holds for all $t\in\mathbb{T}$. \end{rmk}
The solution of system $(\ref{eq3.3})$ satisfying $\left(\begin{matrix}
		x(t_0)\\
		y(t_0)
	\end{matrix}\right)=\left(\begin{matrix}
		x_0\\
		y_0
	\end{matrix}\right)$ can be represented as
	\begin{equation}\label{eqe2}
		\left(\begin{matrix}
		x(t)\\
		y(t)
	\end{matrix}\right)=X(t)\left(\begin{matrix}
		x_0\\
		y_0
	\end{matrix}\right)+\int_{t_0}^tX(t)X^{-1}(s)(I+\mu(s) A(s))^{-1}\left(\begin{matrix}
		0\\
		h(s)y(s)
	\end{matrix}\right)\Delta s.
	\end{equation}
	Note that $$X^{-1}(s)=\left(\begin{matrix}
		\dfrac{\cos_{\phi}(s)}{e_{\mu\phi^2}(s)}&-\dfrac{\sin_{\phi}(s)}{\phi(s)e_{\mu\phi^2}(s)}\\&\\
		\dfrac{\phi(0)\sin_{\phi}(s)}{e_{\mu\phi^2}(s)}&\dfrac{\phi(0)\cos_{\phi}(s)}{\phi(s)e_{\mu\phi^2}(s)}
	\end{matrix}\right),\quad I+\mu(s) A(s)=\left(\begin{matrix}
		1&\mu(s)\\&\\
		-\mu(s)q(s)&\dfrac{\phi^{\sigma}(s)}{\phi(s)}
	\end{matrix}\right),$$
$$\det(I+\mu(s) A(s))=\dfrac{\phi^{\sigma}(s)}{\phi(s)}+\mu^2(s)q(s)=\frac{\phi^{\sigma}(s)(1+\mu^2(s)\phi^2(s))}{\phi(s)},$$
and $$(I+\mu(s) A(s))^{-1}=\left(\begin{matrix}
		\dfrac{1}{1+\mu^2(s)\phi^2(s)}&\dfrac{-\mu(s)\phi(s)}{\phi^{\sigma}(s)(1+\mu^2(s)\phi^2(s))}\\&\\
		\dfrac{\mu(s)\phi^2(s)}{1+\mu^2(s)\phi^2(s)}&\dfrac{\phi(s)}{\phi^{\sigma}(s)(1+\mu^2(s)\phi^2(s))}
	\end{matrix}\right).$$
	 Substituting them in Eq. $(\ref{eqe2})$, then we have\begin{equation}\label{eq3.6}
	\begin{split}
		\left(\begin{matrix}
		x(t)\\
		y(t)
	\end{matrix}\right)&=\left(\begin{matrix}
		\cos_{\phi}(t)&\dfrac{1}{\phi(t_0)}\sin_{\phi}(t)\\&\\
		-\phi\sin_{\phi}(t)&\dfrac{\phi(t)}{\phi(t_0)}\cos_{\phi}(t)
	\end{matrix}\right)\left(\begin{matrix}
		x_0\\
		y_0
	\end{matrix}\right)\\
	 &+\int_{t_0}^t\left(\begin{matrix}
		h(s)\dfrac{-\mu(s)\phi(s)\cos_{\phi}(t,s)+\sin_{\phi}(t,s)}{\phi^{\sigma}(s)(1+\mu^2(s)\phi^2(s))}y(s)\\&\\
		h(s)\dfrac{\mu(s)\phi(s)\phi(t)\sin_{\phi}(t,s)+\phi(t)\cos_{\phi}(t,s)}{\phi^{\sigma}(s)(1+\mu^2(s)\phi^2(s))}y(s)		
	\end{matrix}\right)\Delta s.
	\end{split}
\end{equation}
Let $\left(\begin{matrix}
	x(t)\\y(t)
\end{matrix}\right),
\left(\begin{matrix}
	\bar{x}(t)\\\bar{y}(t)
\end{matrix}\right)$ denote the solutions of system $(\ref{eq3.3})$(i.e. $(\ref{a2})$) that satisfy the initial condition
$\left(\begin{matrix}
	x(t_0)\\y(t_0)
\end{matrix}\right)=\left(\begin{matrix}
	1\\0
\end{matrix}\right)$,
$\left(\begin{matrix}
	\bar{x}(t_0)\\\bar{y}(t_0)
\end{matrix}\right)=\left(\begin{matrix}
	0\\1
\end{matrix}\right)$, respectively. By Eq. $(\ref{eqdab})$  we get\begin{equation}\label{eq3.7}
	\mathcal{A}=x(t_0+T)+\bar{y}(t_0+T).
\end{equation}

	Now let's use the approximation method to calculate $\mathcal{A}$. We assume that
$$\left(\begin{matrix}
	x_0(t)\\y_0(t)
\end{matrix}\right)=X(t)\left(\begin{matrix}
	1\\0
\end{matrix}\right)=\left(\begin{matrix}
		\cos_{\phi}(t)\\
		-\phi(t)\sin_{\phi}(t)
	\end{matrix}\right).$$
	And if $\left(\begin{matrix}
	x_{n-1}(t)\\y_{n-1}(t)
\end{matrix}\right)$ was given, then we define $\left(\begin{matrix}
		x_n(t)\\
		y_n(t)
	\end{matrix}\right)$ inductively by
\begin{equation}\label{eq3.8}
	\left(\begin{matrix}
		x_n(t)\\
		y_n(t)
	\end{matrix}\right)=X(t)\left(\begin{matrix}
		1\\
		0
	\end{matrix}\right)
	 +\int_{t_0}^t\left(\begin{matrix}
		h(s)\dfrac{-\mu(s)\phi(s)\cos_{\phi}(t,s)+\sin_{\phi}(t,s)}{\phi^{\sigma}(s)(1+\mu^2(s)\phi^2(s))}y_{n-1}(s)\\&\\
		h(s)\dfrac{\mu(s)\phi(s)\phi(t)\sin_{\phi}(t,s)+\phi(t)\cos_{\phi}(t,s)}{\phi^{\sigma}(s)(1+\mu^2(s)\phi^2(s))}y_{n-1}(s)		
	\end{matrix}\right)\Delta s.
\end{equation}
Similarly, we  assume that
$$\left(\begin{matrix}
	\bar{x}_0(t)\\\bar{y}_0(t)
\end{matrix}\right)=X(t)\left(\begin{matrix}
	0\\1
\end{matrix}\right)=\left(\begin{matrix}
		\dfrac{1}{\phi(t_0)}\sin_{\phi}(t)\\&\\
		\dfrac{\phi(t)}{\phi(t_0)}\cos_{\phi}(t)
	\end{matrix}\right).$$
	And if $\left(\begin{matrix}
	\bar{x}_{n-1}(t)\\\bar{y}_{n-1}(t)
\end{matrix}\right)$ was given, then we define $\left(\begin{matrix}
		\bar{x}_n(t)\\
		\bar{y}_n(t)
	\end{matrix}\right)$ inductively by
\begin{equation}\label{eq3.9}
	\left(\begin{matrix}
		\bar{x}_n(t)\\
		\bar{y}_n(t)
	\end{matrix}\right)=X(t)\left(\begin{matrix}
		0\\
		1
	\end{matrix}\right)
	 +\int_{t_0}^t\left(\begin{matrix}
		h(s)\dfrac{-\mu(s)\phi(s)\cos_{\phi}(t,s)+\sin_{\phi}(t,s)}{\phi^{\sigma}(s)(1+\mu^2(s)\phi^2(s))}\bar{y}_{n-1}(s)\\&\\
		h(s)\dfrac{\mu(s)\phi(s)\phi(t)\sin_{\phi}(t,s)+\phi(t)\cos_{\phi}(t,s)}{\phi^{\sigma}(s)(1+\mu^2(s)\phi^2(s))}\bar{y}_{n-1}(s)		
	\end{matrix}\right)\Delta s.
\end{equation}
It is easy to see that\begin{equation}\label{eq3.10}
	\left\{\begin{split}
		x_1(t)=&\cos_{\phi}(t)-\displaystyle{\int_{t_0}^th(s)\phi(s)\sin_{\phi}(s)\dfrac{-\mu(s)\phi(s)\cos_{\phi}(t,s)+\sin_{\phi}(t,s)}{\phi^{\sigma}(s)(1+\mu^2(s)\phi^2(s))}\Delta s,}\\
		y_1(t)=&-\phi(t)\sin_{\phi}(t)-\phi(t)\displaystyle{\int_{t_0}^th(s)\phi(s)\sin_{\phi}(s)\dfrac{\mu(s)\phi(s)\sin_{\phi}(t,s)+\cos_{\phi}(t,s)}{\phi^{\sigma}(s)(1+\mu^2(s)\phi^2(s))}\Delta s},\\
		\bar{y}_1(t)=&\dfrac{\phi(t)}{\phi(t_0)}\cos_{\phi}(t)+\dfrac{\phi(t)}{\phi(t_0)}\displaystyle{\int_{t_0}^th(s)\phi(s)cos_{\phi}(s)	\dfrac{\mu(s)\phi(s)\sin_{\phi}(t,s)+\cos_{\phi}(t,s)}{\phi^{\sigma}(s)(1+\mu^2(s)\phi^2(s))}\Delta s}.	
\end{split}\right.
\end{equation}
\begin{rmk}
	Note that $\bar{x}_1(t)$ doesn't work for recursion, so we don't have to figure it out. For the same reason, $\bar{x}_n(t)$ also needn't to be calculated.
\end{rmk}
Let\begin{equation}
	\begin{split}
		P(t,s)=&\dfrac{-\mu(s)\phi(s)\cos_{\phi}(t,s)+\sin_{\phi}(t,s)}{\phi^{\sigma}(s)(1+\mu^2(s)\phi^2(s))},\\
Q(t,s)=&\dfrac{\mu(s)\phi(s)\phi(t)\sin_{\phi}(t,s)+\phi(t)\cos_{\phi}(t,s)}{\phi^{\sigma}(s)(1+\mu^2(s)\phi^2(s))}.
	\end{split}
\end{equation}It can be seen that\begin{equation*}
	\begin{split}
		\sin_{\phi}(\sigma(s),t)=&\frac{e_{i\phi}(\sigma(s),t)-e_{-i\phi}(\sigma(s),t)}{2i}\\=&\frac{(1+i\mu(s)\phi(s))e_{i\phi}(s,t)-(1-i\mu(s)\phi(s))e_{-i\phi}(s,t)}{2i}\\=&\sin_{\phi}(s,t)+\mu(s)\phi(s)\cos_{\phi}(s,t),
	\end{split}
\end{equation*}and $$\sin_{\phi}(t,s)=-e_{\mu\phi^2}(t,s)\sin_{\phi}(s,t).$$Similarly, we have $$\cos_{\phi}(\sigma(s),t)=\cos_{\phi}(s,t)-\mu(s)\phi(s)\sin_{\phi}(s,t),$$ and $$\cos_{\phi}(t,s)=e_{\mu\phi^2}(t,s)\cos_{\phi}(s,t).$$ Then the function $P,Q$ can be simplified as \begin{equation}\label{pq}
	P(t,s)=\frac{1}{\phi^{\sigma}(s)}\sin_{\phi}(t,\sigma(s)),\qquad Q(t,s)=\frac{\phi(t)}{\phi^{\sigma}(s)}\cos_{\phi}(t,\sigma(s)).
\end{equation}$$$$
Using Eq. $(\ref{eq3.8})$, $(\ref{eq3.9})$, $(\ref{eq3.10})$ we obtain
\begin{equation}\label{eq3.11}
	\left\{\begin{split}
		x_2(t)=&\cos_{\phi}(t)-\displaystyle{\int_{t_0}^th(s)\sin_{\phi}(s)P(t,s)\phi(s)\Delta s}
	\\&-\displaystyle{\int_{t_0}^t\int_{t_0}^{t_1}h(t_1)h(t_2)\sin_{\phi}(t_2)P(t,t_1)Q(t_1,t_2)\phi(t_2)\Delta t_2\Delta t_1},\\
	y_2(t)=&-\displaystyle{\phi(t)\sin_{\phi}(t)-\int_{t_0}^th(s)\sin_{\phi}(s)Q(t,s)\phi(s)\Delta s}\\&-\displaystyle{\int_{t_0}^t\int_{t_0}^{t_1}h(t_1)h(t_2)\sin_{\phi}(t_2)Q(t,t_1)Q(t_1,t_2)\phi(t_2)\Delta t_2\Delta t_1,}\\
	\bar{y}_2(t)=&\displaystyle{\dfrac{\phi(t)}{\phi(t_0)}\cos_{\phi}(t)+\dfrac{1}{\phi(t_0)}\int_{t_0}^th(s)\cos_{\phi}(s)Q(t,s)\phi(s)\Delta s}\\&
	+\displaystyle{\dfrac{1}{\phi(t_0)}\int_{t_0}^t\int_{t_0}^{t_1}h(t_1)h(t_2)\cos_{\phi}(t_2)Q(t,t_1)Q(t_1,t_2)\phi(t_2)\Delta t_2\Delta t_1.}
	\end{split}
\right.
\end{equation}
Let\begin{equation}\label{eq3.12}
\left\{\begin{split}
	u_k(t)=&\displaystyle{-\int_{t_0}^t\int_{t_0}^{t_1}\cdots\int_{t_0}^{t_{k-1}}\phi(t_k)\sin_{\phi}(t_k)Q(t_{k-1},t_k)}\\&\displaystyle{\cdots Q(t_1,t_2)P(t,t_1)\prod\limits_{i=1}^kh(t_i)\Delta t_k\cdots\Delta t_1,}\\
	v_k(t)=&\displaystyle{-\int_{t_0}^t\int_{t_0}^{t_1}\cdots\int_{t_0}^{t_{k-1}}\phi(t_k)\sin_{\phi}(t_k)Q(t_{k-1},t_k)}\\&\displaystyle{\cdots Q(t_1,t_2)Q(t,t_1)\prod\limits_{i=1}^kh(t_i)\Delta t_k\cdots\Delta t_1,}\\
	\bar{v}_k(t)=&\displaystyle{\dfrac{1}{\phi(t_0)}\int_{t_0}^t\int_{t_0}^{t_1}\cdots\int_{t_0}^{t_{k-1}}\phi(t_k)\cos_{\phi}(t_k)Q(t_{k-1},t_k)}\\&\displaystyle{\cdots Q(t_1,t_2)Q(t,t_1)\prod\limits_{i=1}^kh(t_i)\Delta t_k\cdots\Delta t_1,}\\&(t_0\leq t_k\leq t_{k-1}\leq\cdots\leq t_1\leq t,~k=1,2,\cdots).
	\end{split}\right.
\end{equation}
For Eq. $(\ref{eq3.10})$, we have\begin{equation}
	\left\{\begin{split}
		x_1(t)=&\cos_{\phi}(t)+u_1(t),\\
		y_1(t)=&-\phi(t)\sin_{\phi}(t)+v_1(t),\\
		\bar{y}_1(t)=&\dfrac{\phi(t)}{\phi(t_0)}\cos_{\phi}(t)+\bar{v}_1(t).
	\end{split}\right.
\end{equation}
For Eq. $(\ref{eq3.11})$, we have \begin{equation}
	\left\{\begin{split}
		x_2(t)=&\cos_{\phi}(t)+u_1(t)+u_2(t),\\
		y_2(t)=&-\phi(t)\sin_{\phi}(t)+v_1(t)+v_2(t),\\
		\bar{y}_2(t)=&\dfrac{\phi(t)}{\phi(t_0)}\cos_{\phi}(t)+\bar{v}_1(t)+\bar{v}_2(t).
	\end{split}\right.
\end{equation}
Now we take an inductive assumption that\begin{equation}\label{eq3.13}
	\left\{\begin{split}
		x_k(t)=&\cos_{\phi}(t)+u_1(t)+\cdots+u_k(t),\\
		y_k(t)=&-\phi(t)\sin_{\phi}(t)+v_1(t)+\cdots+v_k(t),\\
		\bar{y}_k(t)=&\dfrac{\phi(t)}{\phi(t_0)}\cos_{\phi}(t)+\bar{v}_1(t)\cdots+\bar{v}_k(t).
	\end{split}\right.
\end{equation}
According to Eq. $(\ref{eq3.8})$ and $(\ref{eq3.9})$,
\begin{equation}\label{eq3.14}
	\left\{\begin{split}
		x_{k+1}(t)=&\displaystyle{\cos_{\phi}(t)+\int_{t_0}^th(s)P(t,s)y_k(s)\Delta s,}\\
		y_{k+1}(t)=&\displaystyle{-\phi(t)\sin_{\phi}(t)+\int_{t_0}^th(s)Q(t,s)y_k(s)\Delta s,}\\
		\bar{y}_{k+1}(t)=&\displaystyle{\dfrac{\phi(t)}{\phi(t_0)}\cos_{\phi}(t)+\int_{t_0}^th(s)Q(t,s)\bar{y}_k(s)\Delta s.}
	\end{split}\right.
\end{equation}
Substituting Eq. $(\ref{eq3.13})$ into Eq. $(\ref{eq3.14})$, we get
\begin{equation}\label{eq3.15}
	\left\{\begin{split}
		x_{k+1}(t)=&\cos_{\phi}(t)+u_1(t)+\cdots+u_{k+1}(t),\\
		y_{k+1}(t)=&-\phi(t)\sin_{\phi}(t)+v_1(t)+\cdots+v_{k+1}(t),\\
		\bar{y}_{k+1}(t)=&\dfrac{\phi(t)}{\phi(t_0)}\cos_{\phi}(t)+\bar{v}_1(t)+\cdots+\bar{v}_{k+1}(t).
	\end{split}\right.
\end{equation}
This implies that Eq.$(\ref{eq3.13})$ holds for all $k\in\mathbb{N}$.

Let $[t_0,t_0+T]_{\mathbb{T}}:=[t_0,t_0+T]\cap\mathbb{T}$. For the bounded closed interval $[t_0,t_0+T]_{\mathbb{T}}$, consider the series\begin{equation}\label{jishu}
	y_0(t)+\sum\limits_{k=1}^{\infty}[y_k(t)-y_{k-1}(t)],\quad t\in[t_0,t_0+T]_{\mathbb{T}},
\end{equation}and the partial sum  $$y_0(t)+\sum\limits_{k=1}^{n}[y_k(t)-y_{k-1}(t)]=y_n(t).$$ So if we want to prove the sequence $\{y_n(t)\}$ is uniformly convergent on $[t_0,t_0+T]_{\mathbb{T}}$, just show that series $(\ref{jishu})$ converges uniformly on $[t_0,t_0+T]_{\mathbb{T}}$. Note that $\sin_{\phi}(t,s),\cos_{\phi}(t,s),\phi(t), \mu(t),h(t)$ are rd-continuous. By lemma $\ref{bound}$, we have the functions $$|\phi(t)|,|\sin_{\phi}(t)|,|\cos_{\phi}(t)|,|h(t)|$$ are all bounded on compact set $[t_0,t_0+T]_{\mathbb{T}}$. By Eq. $(\ref{pq})$ , since $\phi(t)\neq 0$, it can be seen that $|P(t,s)|,|Q(t,s)|$ are all bounded on $[t_0,t_0+T]_{\mathbb{T}}\times [t_0,t_0+T]_{\mathbb{T}}$. Let $M$ denote their common upper bound, so we have\begin{equation}\label{eq624}
	\begin{split}
		&|y_k(t)-y_{k-1}(t)|=|v_k(t)|\\
		=&\left|\int_{t_0}^t\int_{t_0}^{t_1}\cdots\int_{t_0}^{t_{k-1}}\phi(t_k)\sin_{\phi}(t_k)Q(t_{k-1},t_k)\cdots Q(t_1,t_2)Q(t,t_1)\prod\limits_{i=1}^kh(t_i)\Delta t_k\cdots\Delta t_1\right|\\
		\leq & \int_{t_0}^t\int_{t_0}^{t_1}\cdots\int_{t_0}^{t_{k-1}}\left|\phi(t_k)\sin_{\phi}(t_k)Q(t_{k-1},t_k)\cdots Q(t_1,t_2)Q(t,t_1)\prod\limits_{i=1}^kh(t_i)\right|\Delta t_k\cdots\Delta t_1\\
		\leq & \int_{t_0}^t\int_{t_0}^{t_1}\cdots\int_{t_0}^{t_{k-1}}M^{2k+2}\Delta t_k\cdots\Delta t_1
		\leq \frac{M^{2k+2}(t-t_0)^k}{k!}\leq \frac{M^{2k+2}T^k}{k!},\quad t_0\leq t\leq t_0+T.
	\end{split}
\end{equation}
The third inequality in $(\ref{eq624})$ is derived from Corollary $(\ref{cor624})$.
According to Weierstrass Discriminance, series $(\ref{jishu})$ is uniformly convergent on $[t_0,t_0+T]_{\mathbb{T}}$, thus the sequence $\{y_k(t)\}$ is uniformly convergent on $[t_0,t_0+T]_{\mathbb{T}}$.
Now assume $$\lim\limits_{k\rightarrow\infty}y_k(t)=y^{*}(t).$$ By lemma $\ref{yizhi}$ we get $y^{*}(t)$ is rd-continuous on $[t_0,t_0+T]_{\mathbb{T}}$. Hence\begin{equation}
	\begin{split}
		\lim\limits_{k\rightarrow\infty }y_k(t)&=-\phi(t)\sin_{\phi}(t)+\lim\limits_{k\rightarrow\infty }\int_{t_0}^th(s)Q(t,s)y_{k-1}(s)\Delta s\\&=-\phi(t)\sin_{\phi}(t)+\int_{t_0}^t\lim\limits_{k\rightarrow\infty }h(s)Q(t,s)y_{k-1}(s)\Delta s,
	\end{split}
\end{equation}i.e.,$$y^{*}(t)=-\phi(t)\sin_{\phi}(t)+\int_{t_0}^th(s)Q(t,s)y^{*}(s)\Delta s.
$$
In the same way, the sequence $\{x_k(t)\}$ uniformly converges to $x^*(t)$ which satisfies $$x^*(t)=\cos_{\phi}(t)+\int_{t_0}^th(s)P(t,s)y^*(s)\Delta s.$$ That is to say $\left(\begin{matrix}
	x^*(t)\\y^*(t)
\end{matrix}\right)$ is the solution of system $(\ref{eq3.3})$ with the initial condition $$\left(\begin{matrix}
	x^*(t_0)\\y^*(t_0)
\end{matrix}\right)=\left(\begin{matrix}
	1\\0
\end{matrix}\right).$$ For the theorem of existence and uniqueness of solution, $x^*(t)=x(t),~y^{*}(t)=y(t)$. Let's do the same things for $\bar{y}_n(t)$. Finally we have $\left(\begin{matrix}
	x_n(t)\\y_n(t)
\end{matrix}\right)$ uniformly converges to $\left(\begin{matrix}
	x(t)\\y(t)
\end{matrix}\right)$ and $\bar{y}_n(t)$ uniformly converges to $\bar{y}(t)$.
Let$$\left\{\begin{array}{ccl}
	\mathcal{A}_0&=&x_0(t_0+T)+\bar{y}_0(t_0+T),\\
	\mathcal{A}_1&=&u_1(t_0+T)+\bar{v}_1(t_0+T),\\
	\cdots &&\\
	\mathcal{A}_n&=&u_n(t_0+T)+\bar{v}_n(t_0+T).
\end{array}\right.$$
By $\mathcal{A}=x(t_0+T)+\bar{y}(t_0+T)$ and Eq. $(\ref{eq3.13})$, we get\begin{equation}\label{eq3.16}
	\mathcal{A}=\sum\limits_{n=0}^{\infty}\mathcal{A}_n.
\end{equation}Now we evaluate $\mathcal{A}_n(n=0,1,2,3,\cdots)$:\begin{equation}\label{eq3.19}
	\begin{split}
	\mathcal{A}_0=&\displaystyle{\left(1+\dfrac{\phi(t_0+T)}{\phi(t_0)}\right)\cos_{\phi}(t_0+T)}\\
	    \mathcal{A}_1 =&\displaystyle{\int_{t_0}^{t_0+T}\left(\dfrac{1}{\phi(t_0)}\cos_{\phi}(t_1)Q(t_0+T,t_1)-\sin_{\phi}(t_1)P(t_0+T,t_1)\right)\phi(t_1)h(t_1)\Delta t_1}\\
	    		\mathcal{A}_n =&-\displaystyle{\int_{t_0}^{t_0+T}\int_{t_0}^{t_1}\cdots\int_{t_0}^{t_{n-1}}\phi(t_n)\sin_{\phi}(t_n)Q(t_{n-1},t_n)}\\&\displaystyle{\cdots Q(t_1,t_2)P(t_0+T,t_1)\prod\limits_{i=1}^nh(t_i)\Delta t_n\cdots\Delta t_1}\\&+\displaystyle{\dfrac{1}{\phi(t_0)}\int_{t_0}^{t_0+T}\int_{t_0}^{t_1}\cdots\int_{t_0}^{t_{n-1}}\phi(t_n)\cos_{\phi}(t_n)Q(t_{n-1},t_n)}\\&\displaystyle{\cdots Q(t_1,t_2)Q(t_0+T,t_1)\prod\limits_{i=1}^nh(t_i)\Delta t_n\cdots\Delta t_1}\\=&
		\displaystyle{\int_{t_0}^{t_0+T}\int_{t_0}^{t_1}\cdots\int_{t_0}^{t_{n-1}}\left(\dfrac{1}{\phi(t_0)}\cos_{\phi}(t_n)Q(t_0+T,t_1)-\sin_{\phi}(t_n)P(t_0+T,t_1)\right)\cdot}\\&\displaystyle{\phi(t_n)	Q(t_{n-1},t_n)\cdots Q(t_1,t_2)\prod\limits_{i=1}^nh(t_i)\Delta t_n\cdots\Delta t_1},\quad n\geq2.		\end{split}
\end{equation}
Thus we have\begin{equation}
\begin{split}
	\mathcal{A} =&\displaystyle{\left(1+\dfrac{\phi(t_0+T)}{\phi(t_0)}\right)\cos_{\phi}(t_0+T)}\\&+\displaystyle{\int_{t_0}^{t_0+T}\left(\dfrac{1}{\phi(t_0)}\cos_{\phi}(t_1)Q(t_0+T,t_1)-\sin_{\phi}(t_1)P(t_0+T,t_1)\right)\phi(t_1)h(t_1)\Delta t_1}\\&+\displaystyle{\sum\limits_{n=2}^{\infty}\int_{t_0}^{t_0+T}\int_{t_0}^{t_1}\cdots\int_{t_0}^{t_{n-1}}\left(\dfrac{1}{\phi(t_0)}\cos_{\phi}(t_n)Q(t_0+T,t_1)-\sin_{\phi}(t_n)P(t_0+T,t_1)\right)\cdot}\\&\displaystyle{\phi(t_n)	Q(t_{n-1},t_n)\cdots Q(t_1,t_2)\prod\limits_{i=1}^nh(t_i)\Delta t_n\cdots\Delta t_1}.
\end{split}
	\end{equation}
The formula above can be used for approximations and error estimates. Let$$h(t,s)=\left(\dfrac{1}{\phi(t_0)}\cos_{\phi}(t)Q(t_0+T,s)-\sin_{\phi}(t)P(t_0+T,s)\right)\cdot\phi(t).$$
Then we have $$|\mathcal{A}_n|\leq \displaystyle{\int_{t_0}^{t_0+T}\int_{t_0}^{t_1}\cdots\int_{t_0}^{t_{n-1}}K_1K_2^{n-1}K_3^n\Delta t_n\cdots\Delta t_1}\leq \frac{K_1K_2^{n-1}K_3^nT^n}{n!},$$ where $K_1,K_2,K_3$ are upper bounds of $|h(t,s)|$, $|Q(t,s)|$ and $|h(t)|$ respectively. Let\begin{equation}\label{eqer1}
\mathcal{A}(n)=\mathcal{A}_0+\mathcal{A}_1+\cdots+\mathcal{A}_n,
	\end{equation}
and we have the following error estimate\begin{equation}\label{eqer2}
	|\mathcal{A}-\mathcal{A}(n)|\leq\sum\limits_{k=n+1}^{\infty}\frac{K_1}{K_2}\frac{(K_2K_3T)^k}{k!}=\frac{K_1}{K_2}\left(e^{K_2K_3T}-\sum\limits_{k=0}^{n}\frac{(K_2K_3T)^k}{k!}\right).
\end{equation}
\begin{thm}
	\label{thm12.112}The expression of $\mathcal{A}$ mentioned in Theorem \ref{thm12.111} is \begin{equation}\label{eq3.21}
\begin{split}
	\mathcal{A} =&\displaystyle{\left(1+\dfrac{\phi(t_0+T)}{\phi(t_0)}\right)\cos_{\phi}(t_0+T)}\\&+\displaystyle{\int_{t_0}^{t_0+T}\left(\dfrac{1}{\phi(t_0)}\cos_{\phi}(t_1)Q(t_0+T,t_1)-\sin_{\phi}(t_1)P(t_0+T,t_1)\right)\phi(t_1)h(t_1)\Delta t_1}\\&+\displaystyle{\sum\limits_{n=2}^{\infty}\int_{t_0}^{t_0+T}\int_{t_0}^{t_1}\cdots\int_{t_0}^{t_{n-1}}\left(\dfrac{1}{\phi(t_0)}\cos_{\phi}(t_n)Q(t_0+T,t_1)-\sin_{\phi}(t_n)P(t_0+T,t_1)\right)\cdot}\\&\displaystyle{\phi(t_n)	Q(t_{n-1},t_n)\cdots Q(t_1,t_2)\prod\limits_{i=1}^nh(t_i)\Delta t_n\cdots\Delta t_1},
\end{split}
	\end{equation}and the expression of $\mathcal{B}$ is$$\mathcal{B}=e_{-p+\mu q}(t_0+T,t_0).$$\end{thm}
\begin{thm}\label{thm12.113}
	Let $\mathbb{T}$ be an arbitrary discrete time scale and there are $k$ points in $[t_0,t_0+T)_{\mathbb{T}}$, then equation $(\ref{eq3.21})$ can be simplified as\begin{equation}
\begin{split}
\mathcal{A} =&\mathcal{A}(k)=\displaystyle{\left(1+\dfrac{\phi(t_0+T)}{\phi(t_0)}\right)\cos_{\phi}(t_0+T)}\\&+\displaystyle{\int_{t_0}^{t_0+T}\left(\dfrac{1}{\phi(t_0)}\cos_{\phi}(t_1)Q(t_0+T,t_1)-\sin_{\phi}(t_1)P(t_0+T,t_1)\right)\phi(t_1)h(t_1)\Delta t_1}\\&+\displaystyle{\sum\limits_{n=2}^k\int_{t_0}^{t_0+T}\int_{t_0}^{t_1}\cdots\int_{t_0}^{t_{n-1}}\left(\dfrac{1}{\phi(t_0)}\cos_{\phi}(t_n)Q(t_0+T,t_1)-\sin_{\phi}(t_n)P(t_0+T,t_1)\right)\cdot}\\&\displaystyle{\phi(t_n)	Q(t_{n-1},t_n)\cdots Q(t_1,t_2)\prod\limits_{i=1}^nh(t_i)\Delta t_n\cdots\Delta t_1},
\end{split}
	\end{equation}where $\sum\limits_{n=2}^1(\cdot):=0.$
\end{thm}
\begin{proof}
	Now we show that $\mathcal{A}_n=0$ if $n\geq k+1$, where $\mathcal{A}_n$ is defined in equation $(\ref{eq3.19})$. Let's abbreviate $\mathcal{A}_n$ as $\displaystyle{\int_{t_0}^{t_0+T}\int_{t_0}^{t_1}\cdots\int_{t_0}^{t_{n-1}}R(\cdot)\Delta t_n\cdots\Delta t_1},$ where $t_0\leq t_{n-1}<t_{n-2}<\cdots<t_1<t_0+T$. Note that the number of the points in $ [t_0,t_0+T)_{\mathbb{T}}$ is $k$, which is less than $n$. Hence there must exists an element of the set $\{t_i|i=1,2,\ldots n-1\}$  equal to $t_0$, which implies that $\mathcal{A}_n=0$. The proof is completed.
\end{proof}
\begin{thm}\label{thm12.114}
	Consider the Hill's equation (\cite{Tong, hilles} )\begin{equation}
		\label{hills} x^{\Delta\Delta}(t)+q(t)x(t)=0,
	\end{equation}where $q(t)$ and $\mathbb{T}$ are both $T$-periodic, then the expression of $\mathcal{A}$ of $(\ref{hills})$ can be simplified as\begin{equation}\label{eq11.3}
\begin{split}
	\mathcal{A} =&\displaystyle{\left(1+\dfrac{\phi(t_0+T)}{\phi(t_0)}\right)\cos_{\phi}(t_0+T)}\\&+\displaystyle{\int_{t_0}^{t_0+T}\left(\dfrac{1}{\phi(t_0)}\cos_{\phi}(t_1)Q(t_0+T,t_1)-\sin_{\phi}(t_1)P(t_0+T,t_1)\right)\phi(t_1)h(t_1)\Delta t_1}\\&+\displaystyle{\sum\limits_{n=2}^{\infty}\int_{t_0}^{t_0+T}\int_{t_0}^{t_1}\cdots\int_{t_0}^{t_{n-1}}(-1)^n\left(\dfrac{1}{\phi(t_0)}\cos_{\phi}(t_n)Q(t_0+T,t_1)-\sin_{\phi}(t_n)P(t_0+T,t_1)\right)}\\&\displaystyle{\cdot\phi^{\Delta}(t_1)	\prod\limits_{i=2}^{n}\frac{\sin_{\phi}(t_{i-1},\sigma(t_{i}))\phi^{\Delta}(t_i)}{\phi^{\sigma}(t_i)}\Delta t_n\cdots\Delta t_1}.
\end{split}
	\end{equation}
\end{thm}
\begin{proof}
	The proof is an algebraic process, so we omit it.
\end{proof}
\begin{thm}
	(\cite{shijinlin}) If the time scale $\mathbb{T}=\mathbb{R}$ and $\mathcal{B}=1$, then equation $(\ref{eq3.21})$ can be simplified as$$\mathcal{A}=2\cos\Phi(t_0+T)+\sum\limits_{n=1}^{\infty}\frac{1}{2^{2n-1}}\int_{t_0}^{t_0+T}\int_{t_0}^{t_1}\cdots\int_{t_0}^{t_{2n-1}}\cos\Psi(t_1,\ldots,t_{2n})\cdot\prod\limits_{i=1}^{2n}h(t_i)\mathrm{d}t_{2n}\cdots\mathrm{d}t_1, $$ where $$\displaystyle{\Phi(t)=\int_{t_0}^t\phi(\tau)\mathrm{d}\tau},\quad \displaystyle{\Phi(t,s)=\int_{s}^t\phi(\tau)\mathrm{d}\tau},$$ $$\displaystyle{\Psi(t_1,\cdots,t_{2n})=\Phi(t_0+T)-2\Phi(t_1,t_2)-2\Phi(t_3,t_4)-\cdots-2\Phi(t_{2n-1},t_{2n})}.$$
\end{thm}
\begin{rmk}
	Theoretically, we show that this approach is also valid for critical case: the system has the same characteristic multipliers with modulus equal to one. In a similar manner, we can get an expression of $\bar{x}(t_0+T)$ in the form of a series. That is, combined with the previous discussion, the matrix $\Phi_A(t_0,t_0+T)$ also has an expression in the form of a convergent series. Note that the system we studied in critical case is stable if and only if $\Phi_A(t_0,t_0+T)-\rho I=0,$ where $\rho$ is the characteristic multipliers. Then we can get the error estimate like $(\ref{eqer1})$ and $(\ref{eqer2})$ to analyse the stability. Moreover, we see that the stability of the nonhomogeneous system $x^{\Delta\Delta}+p(t)x^{\Delta}+q(t)x=f(t)$ is equivalent to  the system $x^{\Delta\Delta}+p(t)x^{\Delta}+q(t)x=0$.
\end{rmk}
\section{Program for the algorithm}
The following Matlab program is designed for calculating the value of $\mathcal{A}(n)$ and $\mathcal{B}$ mentioned above. One can  run the following program by Matlab R2018a.\\
{\bf Program 1}
\begin{lstlisting}
   % This program was designed for calculating the value of
	% A(n) and B mentioned in this paper.

	%========================================================
	% Users should set the functions p(t), q(t) and q_diff in
	% advance in section 2 of this script, where q_diff  is
	% the derivative function of q(t) in continuous part(If
	% there is no continuous part, take q_diff=0).
	%========================================================
	% discrete part: Input the discrete points in the  form
	% of a row vector from small to large.
	%========================================================
	% continuous part: Input the ends of continuous intervals
	% in the form of a matrix, and its first and second row
	% record the left and right ends from small to large,
	% respectively.
	clc
	
	global discrete_part continuous_part time_scale;
	discrete_part=input('Enter the discrete point: ');
	continuous_part=input('Enter the continuous interval: ');
	
	if isequal(continuous_part,[])
		time_scale=discrete_part;
	else
		time_scale=sort([discrete_part,continuous_part(1,:),continuous_part(2,:)]);
	end
		B=exp_fun(@(t) -p(t)+mu(t)*q(t),time_scale(end),time_scale(1));
	
	if isequal(continuous_part,[])
		A=valueOfDelta;
		fprintf('The value of A is %f \n',A);
		fprintf('The value of B is %f \n',B);
		fprintf('The modulus of multipliers are %f %f\n',...
			abs((A-sqrt(A^2-4*B))/2),abs((A+sqrt(A^2-4*B))/2));
    elseif isequal(discrete_part,[])
        n=input('n:');
        A=Consum(n);
        fprintf(['The value of A(',num2str(n),') is %f \n'],A);
        fprintf('The value of B is %f \n',B);
        fprintf(['The ',num2str(n),'th approximate modulus are %f %f\n'],...
            abs((A-sqrt(A^2-4*B))/2),abs((A+sqrt(A^2-4*B))/2));
	else
		n=input('n:');
		A=Delta_H(n);
		fprintf(['The value of A(',num2str(n),') is %f \n'],A);
		fprintf('The value of B is %f \n',B);
		fprintf(['The ',num2str(n),'th approximate modulus are %f %f\n'],...
		abs((A-sqrt(A^2-4*B))/2),abs((A+sqrt(A^2-4*B))/2));
	end
		clear global;
	
	%%
	%Users should define the following functions:p,q,q_diff
	function f=p(t)
		if t==pi
			f=0.25;
		else
			f=0;
		end
	end
	
	function f=q(t)
		f=1;
	end
	
	%the derivative function of q(t) in continuous part
	function f=q_diff(t)
		f=0;
	end
	
	%%
	function f=mu(t)
		global discrete_part continuous_part time_scale;
		if ismember(t,discrete_part) || ismember(t,continuous_part(2,:))
			if t==time_scale(end)
				f=mu(time_scale(1));
			else
				for i=1:length(time_scale)
					if t==time_scale(i)
						f=time_scale(i+1)-time_scale(i);
					end
				end
			end
		else
			f=0;
		end
	end
	
	function f=sigma(t)
		f=t+mu(t);
	end
	
	function f=phi(t)
		global discrete_part continuous_part time_scale;
		if isequal(continuous_part,[])
			exphi=NaN(1,length(discrete_part));
			exphi(1)=1;
			for i=2:length(discrete_part)
			   exphi(i)=q(discrete_part(i-1))/exphi(i-1);
		    end
			for i=1:length(discrete_part)
			   if t==discrete_part(i)
				   f=exphi(i);
			   end
			end
			
		else
			leftends=continuous_part(1,:);rightends=continuous_part(2,:);
			if ~(ismember(t,discrete_part) || ismember(t,rightends))
			f=sqrt(q(t));
			
		elseif t<leftends(end)
			n=1;tt=t;
			while ~ismember(tt,leftends)
				n=n+1;tt=sigma(tt);
			end
				temp=NaN(1,n);temp(n)=sqrt(q(tt));k=1;
			while ~isequal(tt,time_scale(k))
				k=k+1;
			end
			for i=n-1:-1:1
				temp(i)=q(time_scale(k-n+i))./temp(i+1);
			end
				f=temp(1);
			
		elseif t==time_scale(end)
			f=phi(time_scale(1));
			
		else
			k=1;
			while ~isequal(t,time_scale(k))
				k=k+1;
			end
				n=length(time_scale)-k+1;
				temp=NaN(1,n);
				temp(n)=phi(time_scale(1));
			for i=n-1:-1:1
				temp(i)=q(time_scale(length(time_scale)-n+i))./temp(i+1);
			end
				f=temp(1);
		end
		end
	end
	
	function f=delta_int(g,t,s)
	% where g is a function handle, t and s are up and low,respectively.
		global continuous_part;ss=s;sum=0;
		if isequal(continuous_part,[])
			while ss<t
				sum=sum+mu(ss).*g(ss);
				ss=sigma(ss);
			end
		else
			rightends=continuous_part(2,:);
			while ss<t
			if ss==sigma(ss)
				k=1;
			while ss>rightends(k)
				k=k+1;
			end
			if rightends(k)>t
				sum=sum+integral(@(x) arrayfun(@(x)g(x),x),ss,t);
			else
				sum=sum+integral(@(x) arrayfun(@(x)g(x),x),ss,rightends(k));
			end
				ss=rightends(k);
		else
			sum=sum+mu(ss).*g(ss);
			ss=sigma(ss);
		end
		end
	end
	f=sum;
	end
	
	function f=cylinder_fun(g,t)
	% where g is a function handle.
	    if mu(t)==0
			f=g(t);
		else
			f=log(1+mu(t).*g(t))./mu(t);
		end
	end
	
	function f=exp_fun(g,t,s)
	% where g is a function handle, t and s are up and low,respectively.
		cylinder_g=@(t)cylinder_fun(g,t);
		f=exp(delta_int(cylinder_g,t,s));
	end
	
	function f=cos_phi(t,s)
		f=(exp_fun(@(x) phi(x).*1i,t,s)+exp_fun(@(x) -phi(x).*1i,t,s))./2;
	end
	
	function f=sin_phi(t,s)
		f=(exp_fun(@(x) phi(x).*1i,t,s)-exp_fun(@(x) -phi(x).*1i,t,s))./2i;
	end
	
	function f=P_H(t,s)
		f=(-mu(s).*phi(s).*cos_phi(t,s)+sin_phi(t,s))./(phi(sigma(s)).*...
		(1+mu(s).^2.*phi(s).^2));
	end
	
	function f=Q_H(t,s)
		f=(mu(s).*phi(s).*phi(t).*sin_phi(t,s)+phi(t).*cos_phi(t,s))./(...
		phi(sigma(s)).*(1+mu(s).^2.*phi(s).^2));
	end
	
	function f=phi_diff(t)
		if mu(t)==0
			f=q_diff(t)/(2*sqrt(q(t)));
		else
			f=(phi(sigma(t))-phi(t))/mu(t);
		end
	end
	
	%need function q_diff(t)
	function f=h_H(t)
		f=-p(t)-phi_diff(t)/phi(t);
	end
	
	function funcn=funvec(n,m)
		global time_scale;
		t_0=time_scale(1);
		T=time_scale(end)-time_scale(1);
		if n==1
			funcn= (1/phi(t_0)*cos_phi(m(n),t_0)*Q_H(t_0+T,m(1))...
			-sin_phi(m(n),t_0)*P_H(t_0+T,m(1)))*phi(m(n))*h_H(m(1));
		else
			last=1;
			for k=2:n
				last=last*Q_H(m(k-1),m(k))*h_H(m(k));
			end
				funcn=last*(1/phi(t_0)*cos_phi(m(n),t_0)*Q_H(t_0+T,m(1))...
				-sin_phi(m(n),t_0)*P_H(t_0+T,m(1)))*phi(m(n))*h_H(m(1));
		end
	end
	
	function f=Delta(n)
		global  time_scale;
		m=time_scale;
		m(end)=[];
		m=sort(m,'descend');
		M=nchoosek(m,n);
		[r,~]=size(M);
		sum=0;
		for i=1:r
			prod=1;
		for j=1:n
			prod=prod*mu(M(i,j));
		end
			sum=sum+prod*funvec(n,M(i,1:n));
		end
			f=sum;
	end
		
	function f=valueOfDelta()
		global time_scale;
		t_0=time_scale(1);
		T=time_scale(end)-time_scale(1);
		sum=(1+phi(t_0+T)/phi(t_0))*cos_phi(t_0+T,t_0);
		for i=1:(length(time_scale)-1)
			sum=sum+Delta(i);
		end
			f=sum;
	end
		
	function f = nIntergrate(fun,n)
		global time_scale;
		t0=time_scale(1);N=n;
		up=cell(1,N);
		up{1}='time_scale(end)';
		for i=2:N
			up{i}=['t',num2str(i-1)];
		end
		expr = GenerateExpr_quadl(N);
		function expr = GenerateExpr_quadl(n)
			if n == 1
				expr = ['delta_int(@(t',num2str(N),')',fun,',',up{N},',t0)'];
			else
				expr = ['delta_int(@(t',num2str(N-n+1),')',...
					GenerateExpr_quadl(n-1),',',up{N-n+1},',t0)'];
			end
		end
		f = eval(expr);
	end
	
	function f=func_ser(n)
		last=['(cos_phi(t',num2str(n),',t0)*Q_H(time_scale(end),t1)/phi(t0)',...
		'-sin_phi(t',num2str(n),',t0)*P_H(time_scale(end),t1))*phi(t',...
		num2str(n),')*h_H(t1)'];
		if n==1
			f=last;
		else
			for i=2:n
				last=[last,'*Q_H(t',num2str(i-1),',t',num2str(i),')*h_H(t',...
				num2str(i),')'];
			end
				f=last;
		end
	end
	
	function f=Delta_H(n)
		global time_scale;
		t0=time_scale(1);
		sum=(1+phi(time_scale(end))/phi(t0))*cos_phi(time_scale(end),t0);
		for i=1:n
			sum=sum+nIntergrate(func_ser(i),i);
		end
			f=sum;
    end
    function f=ConPhi(t,s)
    f=integral(@(x) arrayfun(@(x)sqrt(q(x))+0*x,x),s,t);
    end
    function f=Conh(t)
    f=-p(t)-0.5*q_diff(t)/q(t);
    end
    function f=Confun_sec(n)
    temp='ConPhi(time_scale(end),time_scale(1))';
    temp2='1';
    for i=1:2:n-1
        temp=[temp,'-2*ConPhi(x',num2str(i),',x',num2str(i+1),')'];
    end
    for j=1:n
    temp2=[temp2,'*Conh(x',num2str(j),')'];
    end
    f=['cos(',temp,')','*',temp2];
    end
    function f=Conint_fun_sec(n)
    global B;
    if B==1
     if mod(n,2)==0
        f=ConnIntergrate(Confun_sec(n),n)/(2^(n-1));
     else
        f=0;
     end
    else
        f=ConnIntergrate(Confun_sec(n),n)/(2^(n-1));
    end
    end

    function f=Consum(n)
    global time_scale;
    sum=2*cos(ConPhi(time_scale(end),time_scale(1)));
    for i=1:n
        sum=sum+Conint_fun_sec(i);
    end
    f=sum;
    end
    function f = ConnIntergrate(fun,N)
    global time_scale;
    t0=time_scale(1);
    up=cell(N);low=cell(N);x0=time_scale(end);
    for i=1:N
        low{i}=['t0+0*x',num2str(i-1)];
        up{i}=['x',num2str(i-1)];
    end

        if mod(N,2) == 0
            expr = GenerateExpr_quad2d(N);
        else
            expr = ['quadl(@(x1) arrayfun(@(x1)',GenerateExpr_quad2d(N-1),...
                ',x1),',low{1},',',up{1},')'];
        end
        function expr = GenerateExpr_quad2d(n)
            if n == 2
                expr = ['quad2d(@(x',num2str(N-1),',x',num2str(N),')',...
                    'arrayfun(@(x',num2str(N-1),',x',num2str(N),')',fun,...
                    ',x',num2str(N-1),',x',num2str(N),'),',low{N-1},',',...
                    up{N-1},',@(x',num2str(N-1),')',low{N},',@(x',...
                    num2str(N-1),')',up{N},')'];
            else
                expr = ['quad2d(@(x',num2str(N-n+1),',x',num2str(N-n+2),')',...
                    'arrayfun(@(x',num2str(N-n+1),',x',num2str(N-n+2),')',...
                    GenerateExpr_quad2d(n-2),',x',num2str(N-n+1),',x',...
                    num2str(N-n+2),'),',low{N-n+1},',',up{N-n+1},',@(x',...
                    num2str(N-n+1),')',low{N-n+2},',@(x',num2str(N-n+1),')',...
                    up{N-n+2},')'];
            end
        end
    f = eval(expr);
    end	\end{lstlisting}
     \section{Examples}
\begin{example}(Discrete Time Scale)
	Consider the time scale $\mathbb{T}=\mathbb{Z}$ and the regressive equation \begin{equation}
		\label{eqexam1}
		\Delta\Delta x(t)+\frac{-17+15(-1)^t}{16}\Delta x(t)+\frac{1-15(-1)^t}{16}x(t)=0,
	\end{equation}which can be rewritten as\begin{equation}
		\label{eqexam1.2}\Delta X(t)=\left(\begin{matrix}
			0&1\\&\\
			\displaystyle{-\frac{1-15(-1)^t}{16}}&\displaystyle{-\frac{-17+15(-1)^t}{16}}
		\end{matrix}\right)X(t).	\end{equation}
		Let $A(t)=\left(\begin{matrix}
			0&1\\
			-q(t)&-p(t)
		\end{matrix}\right)=\left(\begin{matrix}
			0&1\\
			-\frac{1-15(-1)^t}{16}&-\frac{-17+15(-1)^t}{16}
		\end{matrix}\right).$ \end{example}
Obviously, the time scale $\mathbb{Z}$ and matrix $A(t)$ have periods of $2$. Also, it can be verified that $\mathcal{B}=e_{-p+\mu q}(2,0)=1$ and then we are going to use formula $(\ref{eq3.21})$ to calculate the value of $\mathcal{A}.$ Taking $$\phi(0)=1,\phi(1)=-\frac{7}{8},\phi(2)=-\frac{8}{7},$$ then we have$$\begin{array}{cccc}
	\cos_{\phi}(0)=1,&\sin_{\phi}(0)=0,\qquad &\cos_{\phi}(1)=1,&\sin_{\phi}(1)=1,\vspace{1ex}\\\displaystyle{\cos_{\phi}(2)=\frac{15}{8}},&\displaystyle{\sin_{\phi}(2)=\frac{1}{8}},\qquad&\displaystyle{\cos_{\phi}(2,1)=1},&\displaystyle{\sin_{\phi}(2,1)=-\frac{7}{8}},\vspace{1ex}\\\displaystyle{P(1,0)=0},&\displaystyle{Q(1,0)=1},\qquad &\displaystyle{P(2,0)=1},&\displaystyle{Q(2,0)=\frac{64}{49}},\vspace{1ex}\\P(2,1)=0,&Q(2,1)=1,\qquad &h(0)=2,&\displaystyle{h(1)=\frac{83}{49}},
\end{array}$$ and
\begin{equation*}
\begin{split}
	\mathcal{A} =&\displaystyle{\left(1+\dfrac{\phi(t_0+T)}{\phi(t_0)}\right)\cos_{\phi}(t_0+T)+\int_0^2\left(\cos_{\phi}(t_1)Q(2,t_1)-\sin_{\phi}(t_1)P(2,t_1)\right)\cdot\phi(t_1)h(t_1)}\Delta t_1\\&+\displaystyle{\int_0^2\int_0^{t_1}\left(\cos_{\phi}(t_2)Q(2,t_1)-\sin_{\phi}(t_2)P(2,t_1)\right)\cdot\phi(t_2)Q(t_1,t_2)h(t_1)h(t_2)}\Delta t_2\Delta t_1\\=&\displaystyle{-\frac{15}{56}+\frac{128}{49}-\frac{7}{8}\cdot\frac{83}{49}+\frac{166}{49}=\frac{17}{4}.}
\end{split}
	\end{equation*}

Now we calculate the value of $\mathcal{A}$ using $(\ref{eqdab})$. It can be seen that the transition matrix of system $(\ref{eqexam1.2})$ is given by\begin{equation}
			\label{eqexam1.3}\Phi_A(t,0)=\left(\begin{matrix}
				\displaystyle{2^t-2^t\int_0^t\frac{5+3(-1)^s}{2^{2s+3}}\Delta s}&\displaystyle{2^t\int_0^t\frac{5+3(-1)^s}{2^{2s+3}}\Delta s}\\&\\
				\displaystyle{2^t-2^t\int_0^t\frac{5+3(-1)^s}{2^{2s+3}}\Delta s-\frac{5+3(-1)^t}{2^{t+3}}}&\displaystyle{2^t\int_0^t\frac{5+3(-1)^s}{2^{2s+3}}\Delta s+\frac{5+3(-1)^t}{2^{t+3}}}
			\end{matrix}\right).
		\end{equation}Then we can obtain that $\mathcal{A}=trace(\Phi_A(2,0))=\frac{17}{4}$, which is consistent with the previous calculations, and we get system $(\ref{eqexam1})$ is unstable. We also can use Program 1 given in Section 5 to calculate:
		\begin{lstlisting}
Enter the discrete point: [0,1,2]
Enter the continuous interval: []
The value of A is 4.250000
The value of B is 1.000000
The modulus of multipliers are 0.250000 4.000000.
		\end{lstlisting}
		\begin{example}
			(Discrete Time Scale) Consider the time scale $\mathbb{T}=2\mathbb{Z}$ and the regressive equation
	\begin{equation}\label{eq20216301}
		\displaystyle{x^{\Delta\Delta}(t)+\frac{\sin\frac{\pi}{3}t+2}{10}x^{\Delta}(t)+\frac{\sin\frac{\pi}{3}t+2}{20}x(t)}=0.
	\end{equation}\end{example}
		  Obviously, the time scale $2\mathbb{Z}$ and the functions $p(t),q(t)$ have periods of $6$. Also, it can be verified that $\mathcal{B}=e_{-p+\mu q}(6,0)=1$. Then we use Program 1 to calculate:
		\begin{lstlisting}
Enter the discrete point: [0,2,4,6]
Enter the continuous interval: []
The value of A is -0.752000
The value of B is 1.000000
The modulus of multipliers are 1.000000 1.000000.
		\end{lstlisting}
	Now we calculate the value of $\mathcal{A}$ using $(\ref{eqdab})$. Let $x_1(t),x_2(t)$ be solutions of $(\ref{eq20216301})$ satisfying$$x_1(0)=1,\quad x_1^{\Delta}(0)=0,\quad x_2(0)=0,\quad x_2^{\Delta}(0)=1.$$ Then we have
	\begin{equation*}
		\begin{array}{lll}
			\displaystyle{x^{\Delta\Delta}_1(0)=-\frac{1}{10},}&\displaystyle{x_1(2)=1,}& \displaystyle{x^{\Delta}_1(2)=-\frac{1}{5},}\vspace{1ex}
			\\\displaystyle{x^{\Delta\Delta}_1(2)=-\frac{3\sqrt{3}-12}{200},}&
			\displaystyle{x_1(4)=\frac{3}{5},}&\displaystyle{x^{\Delta}_1(4)=-\frac{3\sqrt{3}+32}{100},}\vspace{1ex} \\\displaystyle{x_1(6)=-\frac{3\sqrt{3}+2}{50},}&\displaystyle{x^{\Delta\Delta}_2(0)=-\frac{1}{5},}&\displaystyle{x_2(2)=2,}\vspace{1ex}\\\displaystyle{x^{\Delta}_2(2)=\frac{3}{5},}&\displaystyle{x^{\Delta\Delta}_2(2)=-\frac{2\sqrt{3}+8}{25},}&\displaystyle{x_2(4)=\frac{16}{5},}\vspace{1ex}\\\displaystyle{x_2^{\Delta}(4)=-\frac{4\sqrt{3}+1}{25},}&\displaystyle{x_2^{\Delta\Delta}(4)=\frac{55\sqrt{3}-168}{500},}&\displaystyle{x_2^{\Delta}(6)=\frac{15\sqrt{3}-178}{250}.}
		\end{array}
	\end{equation*}
	Thus, $\mathcal{A}=x_1(6)+x_2^{\Delta}(6)=-0.752,$  which is consistent with the previous calculations and we get system $(\ref{eq20216301})$ is stable.

 \begin{example}(Hybrid Time Scale)
	Consider the time scale $\mathbb{T}=[2k\pi,(2k+1)\pi],k\in\mathbb{Z}$ and the regressive equation \begin{equation}\label{eq20216291}
		x^{\Delta\Delta}(t)+p(t)x^{\Delta}(t)+x(t)=0,
	\end{equation}
	 where $$p(t)=\left\{\begin{array}{ll}
		0,& t\in[2k\pi,(2k+1)\pi),\\
		\frac{1}{4},& t=(2k+1)\pi.
	\end{array}\right.$$\end{example}
Obviously, $q(t)=1$ and the time scale $\mathbb{T}$ and the function $p(t)$ have periods of $2\pi$. Also, it can be verified that $\mathcal{B}=e_{-p+\mu q}(2\pi,0)=\pi^2-\frac{\pi}{4}+1$ and then we are going to use formula $(\ref{eq3.21})$ to calculate the value of $\mathcal{A}.$ It can be seen that $\phi(t)=1$ for all $t\in\mathbb{T}$ and $$h(t)=-p(t)-\frac{\phi^{\Delta}(t)}{\phi(t)}=\left\{\begin{array}{ll}
	0,& t\in[2k\pi,(2k+1)\pi),\\
		-\frac{1}{4},& t=(2k+1)\pi.
\end{array}\right.$$Note that $h(t)=0$ for all $t\in[0,\pi)$, then the expression of $\mathcal{A}$ given by $(\ref{eq3.21})$ can be reduced to\begin{equation}
\begin{array}{ccl}

	\mathcal{A} &=&2\cos_1(2\pi)+\displaystyle{\int_{\pi}^{2\pi}\left(\cos_1(t_1)Q(2\pi,t_1)-\sin_1(t_1)P(2\pi,t_1)\right)\cdot}h(t_1)\Delta t_1\\ &=&2\cos_1(2\pi)+\displaystyle{\mu(\pi)\cdot\left(\cos_1(\pi)Q(2\pi,\pi)-\sin_1(\pi)P(2\pi,\pi)\right)\cdot}h(\pi)\\&=&-2+\pi\cdot(-1-0)\cdot(-\frac{1}{4})=\frac{\pi}{4}-2.
\end{array}
	\end{equation}
	Now we calculate the value of $\mathcal{A}$ using $(\ref{eqdab})$. Let $x_1(t),x_2(t)$ be solutions of $(\ref{eq20216291})$ satisfying$$x_1(0)=1,\quad x_1^{\Delta}(0)=0,\quad x_2(0)=0,\quad x_2^{\Delta}(0)=1.$$ For any $t\in[0,\pi]$, we have $$x_1(t)=\cos t\quad and \quad x_2(t)=\sin t.$$Hence, we get $x_1^{\Delta}(\pi)=0$, $x_2^{\Delta}(\pi)=-1$ and$$x_2^{\Delta\Delta}(\pi)=\frac{x_2^{\Delta}(2\pi)-x_2^{\Delta}(\pi)}{\pi}=-p(\pi)x_2^{\Delta}(\pi)-x_2(\pi).$$Thus, $x_1(2\pi)=-1$, $x_2^{\Delta}(2\pi)=\frac{\pi}{4}-1.$ Finally, we have$$\mathcal{A}=x_1(2\pi)+x_2^{\Delta}(2\pi)=\frac{\pi}{4}-2,$$ which is consistent with the previous calculations and system $(\ref{eq20216291})$ is unstable. We also can use Program 1 given in Section 5 to calculate $\mathcal{A}(n)$ given by $(\ref{eq3.19})$:
\begin{lstlisting}
Enter the discrete point: [2*pi]
Enter the continuous interval: [0;pi]
n:1
The value of A(1) is -1.214602
The value of B is 10.084206
The 1th approximate modulus are 3.175564 3.175564.
\end{lstlisting}
\begin{example}
	(Continuous Time Scale) Consider the time scale $\mathbb{T}=\mathbb{R}$ and the equation \begin{equation}\label{eq9.18}
		x'(t)+\frac{1}{2}\sin(2t)x'(t)+\frac{1}{4}x(t)=0.
	\end{equation}
\end{example}
We can use Program 1 to calculate $\mathcal{A}(n)$:
\begin{lstlisting}
Enter the discrete point: []
Enter the continuous interval: [0;pi]
n:3
The value of A(3) is -0.065450
The value of B is 1.000000
The 3th approximate modulus are 1.000000 1.000000.
\end{lstlisting}
Now we estimate the value of $|\mathcal{A}(3)-\mathcal{A}|$ by $(\ref{eqer2})$. A straightforward calculation leads to $$|\mathcal{A}(3)-\mathcal{A}|\leq e^{\frac{\pi}{2}}-\left(1+\frac{\pi}{2}+\frac{(\frac{\pi}{2})^2}{2}+\frac{(\frac{\pi}{2})^3}{6}\right)\approx 0.360016406528039.$$ It is clear that $0<\mathcal{A}<0.5$. It follows from Theorem \ref{rmkab2} that system $(\ref{eq9.18})$ is stable.
\begin{example}(\cite{tableshi})
	Consider Mathieu equation\begin{equation}
		\label{eqmath}x''+(\lambda-h\cos2t)x=0.
	\end{equation}
\end{example}Book \cite{book2} gets the approximate values of some eigenvalue of $(\ref{eqmath})$ as follows:\par\begin{table}[!h]\centering
	\begin{tabular}{lllll}
		\hline
		h&$\lambda$ & & &\\
		 \cline{2-5}
		 & $\lambda_1$&$\lambda_2$&$\lambda_1'$&$\lambda_2'$\\
		 \hline
		 1&3.979&4.101&9.014&9.018\\
		 2&3.917&4.371&9.047&9.078\\
		 3&3.814&4.747&9.093&9.193\\
		 \hline \vspace{-4pt}
	\end{tabular}
\end{table}
	Let $\mathcal{A}[\lambda_i]$ and $\mathcal{A}[\lambda_i']$ be the value of $\mathcal{A}$ of $(\ref{eqmath})$ as $\lambda=\lambda_i$ and $\lambda=\lambda_i'$, respectively. It is well known that $\mathcal{A}[\lambda_i]\approx 2$ and $\mathcal{A}[\lambda_i']\approx -2.$ Now we are going to calculate the 3-th approximate value of $\mathcal{A}[\lambda_i]$ and $\mathcal{A}[\lambda_i']$ by Program 1 and the results are shown in Table \ref{table1}.\par\begin{table}[!h]\caption{3-th approximate value of $\mathcal{A}$}\label{table1}\centering
	\begin{tabular}{lr@{}l}
		\hline
		Equation& \multicolumn{2}{c}{3-th approximate value of $\mathcal{A}$}\\
		\hline
		$x''+(3.979-\cos2t)x=0$& &2.000049 \\
		$x''+(4.101-\cos2t)x=0$& &2.000044 \\
		$x''+(9.014-\cos2t)x=0$& -&2.000001 \\
		$x''+(9.018-\cos2t)x=0$& -&2.000000 \\
		&\\
		$x''+(3.917-2\cos2t)x=0$& &2.000798 \\
		$x''+(4.371-2\cos2t)x=0$& &2.000384 \\
		$x''+(9.047-2\cos2t)x=0$& -&2.000009 \\
		$x''+(9.078-2\cos2t)x=0$& -&2.000018 \\
		&\\
		$x''+(3.814-3\cos2t)x=0$& &1.998646 \\
		$x''+(4.747-3\cos2t)x=0$& &1.998733 \\
		$x''+(9.093-3\cos2t)x=0$& -&2.000103 \\
		$x''+(9.193-3\cos2t)x=0$& -&2.000093 \\
		\hline\end{tabular}\vspace{-4pt} \end{table}
\newpage
 \section{Declarations section}

 \subsection{Ethical Approval}  Not Applicable.

 \subsection{Availability of supporting data}
 No data was used for the research in this article.

 \subsection{Funding}
     This paper was jointly supported from the National Natural
Science Foundation of China under Grant (No. 11931016, 11671176).

\subsection{Competing interests}
The authors declare that they have no conflict of interest.

\subsection{Authors' contributions}
Yonghui Xia conceived of the study, supervision, participated in the computations and proof, wrote the manuscript text.  Mengda Wu carried out the program and participated in the computations and proof. Ziyi Xu participated in the program.

\subsection{Acknowledgement}
Not Applicable.

\end{document}